\documentclass[11pt]{amsart}

\usepackage[vmargin=2.5cm, hmargin= 2.5cm]{geometry}
\usepackage{enumerate}
\usepackage{xcolor}

\usepackage{tikz}
\usepackage{pgfplots}

\newtheorem{theorem}{Theorem}
\newtheorem{lemma}[theorem]{Lemma}
\newtheorem{corollary}[theorem]{Corollary}
\newtheorem{proposition}[theorem]{Proposition}

\theoremstyle{definition}
\newtheorem{example}[theorem]{Example}

\theoremstyle{remark}
\newtheorem{remark}[theorem]{Remark}

\title[Complete intersections with fixed Frobenius number]{Constructing the set of complete intersection numerical semigroups with a given Frobenius number}

\author{A. Assi}
\address{Universit\'e d'Angers, D\'epartement de Math\'ematiques, LAREMA, UMR 6093, 2 bd Lavoisier, 49045 Angers Cedex 01, France}
\email{assi@univ-angers.fr}

\thanks{The first author is partially supported by the project GDR CNRS 2945}

\author{P. A. Garc\'{\i}a-S\'anchez}
\address{Departamento de \'Algebra, Universidad de Granada, E-18071 Granada, Espa\~na}
\email{pedro@ugr.es}

\thanks{The second author is supported by the projects MTM2010-15595, FQM-343,  FQM-5849, and FEDER funds. This research was performed while the second author visited the Universit\'e d'Angers as invited lecturer, and he wants to thank the D\'epartement de Math\'ematiques of this university for its kind hospitality. The authors would like to thank the referees for their suggestions, comments and examples provided}

\begin{document}
\begin{abstract}
Delorme suggested that the set of all complete intersection numerical semigroups can be computed recursively. We have implemented this algorithm, and particularized it to several subfamilies of this class of numerical semigroups: free and telescopic numerical semigroups, and numerical semigroups associated to an irreducible plane curve singularity. The recursive nature of this procedure allows us to give bounds for the embedding dimension and for the minimal generators of a semigroup in any of these families. 
\end{abstract}

\maketitle

\section{Introduction}

Let $\mathbb N$ denote the set of nonnegative integers. A \emph{numerical semigroup} $\Gamma$ is a submonoid of $\mathbb N$ with finite complement in $\mathbb N$ (this condition is equivalent to $\gcd(\Gamma)=1$). If $\Gamma$ is a numerical semigroup, the elements in $\mathbb N\setminus \Gamma$ are the \emph{gaps} of $\Gamma$. The  cardinality of $\mathbb N\setminus \Gamma$ is the \emph{genus} of $\Gamma$, $\mathrm g(\Gamma)$. The largest integer not in $\Gamma$ is called the \emph{Frobenius number} of $\Gamma$, and will be denoted by $\mathrm F(\Gamma)$. Clearly, $\mathrm F(\Gamma)+1 + \mathbb N\subseteq \Gamma$, and this is why $\mathrm c(\Gamma)=\mathrm F(\Gamma)+1$ is known as the \emph{conductor} of $\Gamma$. 

Since for every $x\in \Gamma$, $\mathrm F(\Gamma)-x$ cannot be in $\Gamma$, we deduce that $\mathrm g(\Gamma)\ge \frac{\mathrm c(\Gamma)}2$. We say that $\Gamma$ is \emph{symmetric} when the equality holds, or equivalently, for every integer $x$, $x\not\in \Gamma$ implies $\mathrm F(\Gamma)-x\in \Gamma$. In this setting,  $\mathrm c(\Gamma)$ is an even integer, and thus $\mathrm F(\Gamma)$ is odd.  

It can be easily proved that any numerical semigroup admits a unique \emph{minimal generating system} (every element is a linear combination of elements in this set with nonnegative integer coefficients and none of its proper subsets fulfills this condition; see for instance \cite[Chapter 1]{ns-book}). If $A=\{r_0,\ldots,r_h\}$ is the minimal generating set of $\Gamma$, then its elements are called \emph{minimal generators}, and its cardinality is the \emph{embedding dimension} of $\Gamma$, $\mathrm e(\Gamma)$. The smallest minimal generator is the smallest positive integer belonging to the semigroup, and it is known as the \emph{multiplicity} of $\Gamma$, denoted by $\mathrm m(\Gamma)$.


The map 
\[ \mathbb N^{\mathrm e(\Gamma)}\to \Gamma,\ \varphi(a_0,\ldots,a_h)=a_0r_0+\cdots +a_hr_h\]
is a monoid epimorphism. Hence $\Gamma$ is isomorphic to $\mathbb N^{\mathrm e(\Gamma)}/\ker \varphi$, where $\ker \varphi=\{ (a,b)\in \mathbb N^{\mathrm e(\Gamma)}\times \mathbb N^{\mathrm e(\Gamma)} ~|~ \varphi(a)=\varphi(b)\}$ ($\ker \varphi$ is a congruence on $\mathbb N^{\mathrm e(\Gamma)}$). 

A  \emph{presentation} for $\Gamma$ is a set of generators of the congruence $\varphi$, and a \emph{minimal presentation} is a set of generators minimal with respect to set inclusion (actually, in our setting also with respect to cardinality; see \cite[Corollary 8.13]{ns-book}). It can be shown that the cardinality of any minimal presentation is greater than or equal to $\mathrm e(\Gamma)-1$, \cite[Theorem 9.6]{ns-book}. A numerical semigroup is a \emph{complete intersection} if this equality holds. 

Given $A$ a set positive integers, and $A=A_1\cup A_2$ a non trivial partition of $A$, we say that $A$ is the \emph{gluing} of $A_1$ and $A_2$ if $\mathrm{lcm}(d_1,d_2)\in \langle A_1\rangle \cap \langle A_2\rangle$, where $d_i=\gcd(A_i)$ and $\langle A_i\rangle$ denotes the monoid generated by $A_i$, $i=1,2$. If $A$ is the minimal system of generators of $\Gamma$, and $\Gamma_i$ is the numerical semigroup generated by $A_i/d_i$, $i=1,2$, we also say that $\Gamma$ is the \emph{gluing} of $\Gamma_1$ and $\Gamma_2$. It turns out that $d_1\in\Gamma_2$, $d_2\in\Gamma_1$, $\gcd(d_1,d_2)=1$, and neither $d_1$ is a minimal generator of $\Gamma_2$ nor $d_2$ is a minimal generator of $\Gamma_1$ (\cite[Section 8.3]{ns-book}). Delorme proved in \cite[Proposition 9]{delorme} that a numerical  semigroup is a complete intersection if and only if it is a gluing of two complete intersection numerical semigroups (though with a different notation; the concept of gluing was introduced in \cite{gluing}). The gluing of symmetric numerical semigroups is symmetric (\cite[Proposition 10 (iii)]{delorme}), and as a consequence of this, complete intersections are symmetric.  

In \cite{fundamental-gaps} there is a procedure to construct the set of all numerical semigroups with a given Frobenius number. We show in this manuscript how can we use the concept of gluing to compute the set of all complete intersection numerical semigroups with a given Frobenius number (or equivalently with fixed genus). Recently there have been some experimental results that point out to the possibility that the number of numerical semigroups with a fixed genus has a Fibonacci like behaviour (\cite{bras}). Indeed, it is known that asymptotically the number of numerical semigroups with given genus grows as the Fibonacci sequence (\cite{zhai}). However there is not a proof for this for all genus, and  we still do not even have a demonstration that there are more numerical semigroups with genus $g+1$ than numerical semigroups with genus $g$. This is not the case for complete intersection numerical semigroups, as we see in the last section.

We also show how to calculate the set of all free (in the sense of \cite{bertin}) numerical semigroups, which is a special subclass of complete intersections, the set of all telescopic numerical semigroups (contained in the set of free numerical semigroups), and that of numerical semigroups associated to an irreducible plane curve singularity (these are a particular case of telescopic numerical semigroups).

The recursive nature of gluing also allows us to give some bounds for the generators and embedding dimension for these families of semigroups when we fix the Frobenius number. The deeper we go in the chain of inclusions given in the preceding paragraph, the smaller are the bounds.

\section{The Frobenius number and multiplicity of a complete intersection}
Let $\Gamma$ be a numerical semigroup. We know that $\Gamma$ is a complete intersection if and only if it is the gluing of two complete intersections. Delorme (though with a different notation) highlighted in \cite[Section 11]{delorme} that this fact can used to determine if a numerical semigroup is a complete intersection (this idea has already been exploited in \cite{bermejo}; and in \cite{free} one can find a procedure to determine if an affine semigroup is the gluing of two affine semigroups), and also to compute the set of all complete intersections. In order to construct the set of all complete intersection numerical semigroups with given Frobenius number, we can proceed recursively by using the following formula for the Frobenius number of the gluing of two numerical semigroups, which is just a reformulation of Delorme's description of the conductor of a gluing.

\begin{proposition}\label{frob-gluing}
Assume that $\Gamma$ is a numerical semigroup minimally generated by $A=A_1\cup A_2$, and that $A$ is the gluing of $A_1$ and $A_2$. Let $d_1=\gcd(A_1)$ and $d_2=\gcd(A_2)$. Define $\Gamma_1=\langle A_1/d_1\rangle$ and $\Gamma_2=\langle A_2/d_2\rangle$. Then 
\[
\mathrm F(\Gamma)= d_1\mathrm F(\Gamma_1)+d_2\mathrm F(\Gamma_2)+ d_1d_2.
\]
\end{proposition}
\begin{proof}
Observe that $\Gamma = d_1 \Gamma_1+ d_2\Gamma_2$. By \cite[Proposition 10 (i)]{delorme}, 
\begin{equation}\label{formula-c} 
 \mathrm c(\Gamma)=d_1 \mathrm c(\Gamma_1) + d_2 \mathrm c(\Gamma_2)+(d_1-1)(d_2-1).
\end{equation}
Having in mind the relationship between Frobenius number and conductor, the formula follows easily.
\end{proof}

%
%

In Proposition \ref{frob-gluing}, $\Gamma= d_1\Gamma_1+d_2\Gamma_2$ and $d=d_1d_2=d_1\Gamma_1\cap d_2\Gamma_2$. The integer $d$ is the element where the gluing takes place. If we repeat the process with $d_1\Gamma_1$ and $d_2\Gamma_2$ in this result, we construct a decomposition tree of $\Gamma$, whose leaves are copies isomorphic to of $\mathbb N$ (this was the idea followed in \cite{b-g-r-v}). Assume that $d^{(1)},\ldots,d^{(h)}$ are the elements where the gluings take place in this splitting. The Frobenius number of $\Gamma$ is precisely $\sum_{i=1}^h d^{(i)} -\sum_{a\in A}a$ (see \cite[Section 11]{delorme} where it is highlighted that this formula is a particular case of a result given in \cite{herzog-kunz}).

\begin{example}
Let $\Gamma=\langle 10,14,15,21\rangle$. Then $\Gamma= \langle 10, 15\rangle + \langle 14,21\rangle$ and $35=5\times 7\in \langle 10,15\rangle$. We repeat the process for $\langle 10,15\rangle=\langle 10\rangle + \langle 15\rangle$ and $\langle 14,21\rangle= \langle 14\rangle +\langle 21\rangle$. We get $30\in \langle 10\rangle \cap \langle 15\rangle$ and $42\in \langle 14\rangle \cap \langle 21\rangle$. Hence the gluings take place at $35$, $30$ and $42$. Thus $\mathrm F(\Gamma)= (35+30+42)-(10+14+15+21)=47$. 
\end{example}

\begin{example}
We construct a complete intersection numerical semigroup with four generators, by gluing two embedding dimension two numerical semigroups.
\begin{verbatim}
gap> s:=NumericalSemigroup(10,11);;
gap> t:=NumericalSemigroup(7,9);;
gap> g:=NumericalSemigroup(16*10,16*11,21*7,21*9);;
gap> FrobeniusNumber(g);
2747
gap> 16*FrobeniusNumber(s)+21*FrobeniusNumber(t)+16*21;
2747
\end{verbatim}
\end{example}

\begin{remark}
For $\Gamma=\mathbb N$, we have 
\[\mathrm c(\mathbb N)=0,\quad  \mathrm F(\mathbb N)=-1,\quad  \mathrm g(\mathbb N)=0,\quad  \mathrm m(\mathbb N)=1, \quad  \mathrm e(\mathbb N)=1.
\]
\end{remark}

\begin{proposition}\label{lower-bound-m-complete-intersection}
If $\Gamma$ is a complete intersection, then 
\[\mathrm m(\Gamma)\ge 2^{\mathrm e(\Gamma)-1}.\]
\end{proposition}
\begin{proof}
Let $h=\mathrm e(\Gamma)-1$. We use induction on $h$. For $h=1$, the statement follows trivially. As $\Gamma$ is a complete intersection, if $A$ is its minimal set of generators, we can find a partition of $A=A_1\cup A_2$ such that $A$ is the gluing of $A_1$ and $A_2$. Set as above $d_i=\gcd(A_i)$, and $\Gamma_i=\langle A_i/d_i\rangle$. Let $h_i=\mathrm e(\Gamma_i)-1$. Hence $h=h_1+h_2+1$. By induction hypothesis $\mathrm m(\Gamma_i)\ge 2^{h_i}$. Recall that $d_1\in \Gamma_2$ and  $d_2\in \Gamma_1$, and they are not minimal generators. Thus $d_1\ge 2\mathrm m(\Gamma_2)\ge 2^{h_2+1}$, and analogously $d_2\ge 2^{h_1+1}$. For every $a\in A_1$, $a/d_1$ is a minimal generator of $\Gamma_1$, whence $a/d_1\ge 2^{h_1}$. Therefore $a\ge 2^{h_1+h_2+1}=2^h$. The same argument shows that any element in $A_2$ is greater than or equal to $2^h$. 
\end{proof}

\begin{example}\label{example-recursive-family}
We construct recursively a family $\{\Gamma^{(n)}\}_{n\in \mathbb N}$ of complete intersection numerical semigroups  reaching the bound of Proposition \ref{lower-bound-m-complete-intersection}. 

We start with $\Gamma^{(1)}=\langle 2,3\rangle$, and the general element in the sequence is defined as  $\Gamma^{(n+1)}= 2\Gamma^{(n)}+(2^{n+1}+1)\mathbb N$. 

For instance, $\Gamma^{(2)}=2\langle 2,3\rangle+ 5\mathbb N=\langle 4,5,6\rangle$, $\Gamma^{(3)}=2\langle 4,5,6\rangle + 9\mathbb N= \langle 8,9,10,12\rangle$, and so on.

It is not hard to prove that 
\begin{multline*}
\Gamma^{(n+1)}=\langle 2^{n+1}, 2^{n+1}+1, 2^{n+1}+2, 2^{n+1}+2^2,\ldots, 2^{n+1}+2^{n}\rangle\\= 2\langle 2^n, 2^n+1,\ldots, 2^n+2^{n-1}\rangle +(2^{n+1}+1)\mathbb N.
\end{multline*}

Notice that $\Gamma^{(n+1)}$ is a gluing of $\Gamma^{(n)}$ and $\mathbb N$, since 
\begin{itemize}
\item $2\in \mathbb N$ and $2$ is not a minimal generator of $\mathbb N$,
\item $2^{n+1}+1$ is the sum of the two smallest minimal generators of $\Gamma^{(n)}$; thus $2^{n+1}+1$ belongs to $\Gamma^{(n)}$ and it is not a minimal generator of $\Gamma^{(n)}$,
\item $\gcd(2,2^{n+1}+1)=1$.
\end{itemize} 

It follows that $\mathrm m(\Gamma^{(n)})=2^n$ and $\mathrm e(\Gamma^{(n)})=n+1$. Thus the bound in Proposition \ref{lower-bound-m-complete-intersection} is attained.
\end{example}

\begin{corollary}\label{upper-bound-ed-complete-intersections}
If $\Gamma$ is a complete intersection numerical semigroup other than $\mathbb N$, then 
\[\mathrm e(\Gamma) \le \log_2(\mathrm c(\Gamma))+1.\] 
\end{corollary}
\begin{proof}
By Proposition \ref{lower-bound-m-complete-intersection}, $2^{\mathrm e(\Gamma)-1}\le \mathrm m(\Gamma)$. Since $\Gamma\neq \mathbb N$, we have that $\mathrm m(\Gamma)\le \mathrm c(\Gamma)$, and the bound follows.
\end{proof}

\begin{remark}
Notice that in the proof of Corollary \ref{upper-bound-ed-complete-intersections} we use $\mathrm m(\Gamma)\le \mathrm c(\Gamma)$. For $\Gamma=\langle 2,3\rangle$, we get an equality and also the bound given in this corollary is reached. If $\mathrm m(\Gamma)=\mathrm c(\Gamma)$, then $\Gamma=\langle m, m+1,\ldots, 2m-1\rangle$, with $m=\mathrm m(\Gamma)$. Hence $\mathrm e(\Gamma)=m$, that is, $\Gamma$ has maximal embedding dimension (it is easy to see that the embedding dimension of a numerical semigroup is always less than or equal to its multiplicity; see for instance \cite[Chapter 1]{ns-book}). It is  well known that the cardinality of a minimal presentation of a maximal embedding dimension numerical semigroup with multiplicity $m$ is $\frac{m(m-1)}2$ (see for instance \cite[Corollary 8.27]{ns-book}). Hence a maximal embedding dimension numerical semigroup with multiplicity $m$ is a complete intersection if and only if $\frac{m(m-1)}2=m-1$, or equivalently, either the numerical semigroup is $\mathbb N$ or $m=2$. If in addition we impose that the conductor and the multiplicity agree, then the only two possibilities are $\mathbb N$ and $\langle 2,3\rangle$. 

From the definitions of multiplicity and conductor, it is easy to see that there is no numerical semigroup $\Gamma$ such that $\mathrm c(\Gamma)=1+\mathrm m(\Gamma)$. 

If $\mathrm c(\Gamma)=2+\mathrm m(\Gamma)$, then 
$\Gamma = \langle m, m+2,m+3,\ldots, 2m-1,2m+1\rangle$, which is a maximal embedding dimension numerical semigroup. So the only complete intersection with $\mathrm c(\Gamma)=2+\mathrm m(\Gamma)$ is $\langle 2,5\rangle$.

The case $\mathrm c(\Gamma)=3+\mathrm m(\Gamma)$ requires more effort. In this setting $m=\mathrm m(\Gamma)>2$. We have two possibilities.
\begin{itemize}
\item $\Gamma= \langle m, m+3,m+4,\ldots, 2m-1,2m+1,2m+2\rangle$, which has maximal embedding dimension, and so it cannot be a complete intersection numerical semigroup, because $m>2$.
\item $\Gamma= \langle m, m+1, m+3,m+4,\ldots, 2m-1\rangle$. Here $\mathrm e(\Gamma)= m-1$ and the minimum element in $\Gamma$ congruent with 2 modulo $m$ is $2m+1=(m+1)+(m+1)$. Thus  in view of \cite[Theorem 1(2)]{high-ed}, the cardinality of a minimal presentation for $\Gamma$ is $\frac{(m-1)(m-2)}{2}$. We conclude that  $\Gamma$ is a complete intersection if and only if $\frac{(m-1)(m-2)}2=m-2$, and as $m>2$, this is equivalent to $m=3$. Hence $\Gamma=\langle 3,4\rangle$.
\end{itemize}

Therefore, if we assume that $\Gamma\not\in \{\mathbb N, \langle 2,3\rangle, \langle 2,5\rangle, \langle 3,4\rangle\}$, and $\Gamma$ is a complete intersection numerical semigroup, then we can assert that $\mathrm c(\Gamma)\ge \mathrm m(\Gamma)+4$, and the bound in Corollary \ref{upper-bound-ed-complete-intersections} can be slightly improved to 
\[
\mathrm e(\Gamma) \le \log_2(\mathrm c(\Gamma)-4)+1.
\]
This bound is attained for instance by $ \langle 2, 7 \rangle$, $\langle  4, 5, 6 \rangle$ and $\langle 4, 6, 7 \rangle$. 

By using \cite[Section 1.2]{high-ed}, we can determine those complete intersections with $\mathrm c(\Gamma)=\mathrm m(\Gamma)+4$, and thus obtain another small improvement of the above bound.
\end{remark}

We can improve this bound by using a different strategy.

\begin{proposition}\label{lower-bound-c-ci}  
Let $\Gamma$ be a complete intersection numerical semigroup. Then \[(\mathrm e(\Gamma)-1) 2^{\mathrm e(\Gamma)-1}\le \mathrm c(\Gamma).\]
\end{proposition}
\begin{proof} 
We use induction on the embedding dimension of $\Gamma$. If the embedding dimension of $\Gamma$ is either one or two, then the result holds trivially. So assume that $\mathrm e(\Gamma)\ge 3$. As $\Gamma$ is a complete intersection, we know that there exist two complete intersection numerical semigroups $\Gamma_1$ and $\Gamma_2$ such that $\Gamma$ is the gluing of $\Gamma_1$ and $\Gamma_2$. Thus there exist $d_1\in \Gamma_2$ and $d_2\in \Gamma_1$, that are not minimal generators, such that $\Gamma=d_1\Gamma_1+d_2\Gamma_2$. For sake of simplicity write $c=\mathrm c(\Gamma)$, $e=\mathrm e(\Gamma)$, $c_i=\mathrm c(\Gamma_i)$ and $e_i=\mathrm e(\Gamma_i)$, $i=1,2$. Then from the definition of gluing we already know that $e=e_1+e_2$.

Since $e\ge 3$, we may assume without loss of generality that $e_1\ge 2$. As $d_1$ is not a minimal generator of $\Gamma_2$, $d_1\ge 2\mathrm m(\Gamma_2)$,  and as $e_1\ge 2$ and $d_2$ is not a minimal generator of $\Gamma_1$,  $d_2\ge 2\mathrm m(\Gamma_1)+1$. In view of Proposition \ref{lower-bound-m-complete-intersection}, we deduce $d_1\ge 2^{e_2}$ and $d_2\ge 2^{e_1}+1$.

Now from \eqref{formula-c}, we have $c=d_1c_1+d_2c_2+(d_1-1)(d_2-1)$. By induction hypothesis and the preceding paragraph, we get $c\ge 2^{e_2}(e_1-1)2^{e_1-1}+(2^{e_1}+1)(e_2-1)2^{e_2-1}+(2^{e_2}-1)2^{e_1} = 
(e-2)2^{e-1}+ (e_2-1)2^{e_2-1}+ 2^e-2^{e_1} \ge (e-1)2^{e-1}-2^{e-1}+2^e-2^{e_1}= (e-1)2^{e-1}+2^{e-1}-2^{e_1}\ge (e-1)2^{e-1}$.
\end{proof}

\begin{example}\label{recursive-family-ii}
Let $\{\Gamma^{(n)}\}_{n\in \mathbb N}$ be the family of numerical semigroups presented in Example \ref{example-recursive-family}. By using \eqref{formula-c}, it is not hard to check inductively that $\mathrm c(\Gamma^{(n)})= n2^n$, and thus the bound of Proposition \ref{lower-bound-c-ci} is attained.

If we have a closer look at the proof of Proposition \ref{lower-bound-c-ci}, then we easily deduce that for the bound to be attained, the following must hold in all induction steps with $e\ge 3$:
\begin{itemize}
\item $(e_2-1)2^{e_2-1}=0$ and thus $e_2=1$, that is, $\Gamma_2$ is $\mathbb N$ (we will study these semigroups in the next section);
\item from $e_2=1$ it follows that $e_1=e-1$ and $2^{e-1}-2^{e_1}=0$;
\item $\mathrm m(\Gamma_1)=2^{e_1-1}$ and $d_2=2\mathrm m(\Gamma_1)+1=2^{e_1}+1$, whence $\mathrm m(\Gamma_1)+1\in \Gamma_1$;
\item $c_1=(e_1-1)2^{e_1-1}=(e-2)2^{e-2}$;
\item $d_1=2$.
\end{itemize}
Also the only embedding dimension two numerical semigroup for which the equality holds is $\langle 2,3\rangle$. If follows that the family given in Example \ref{example-recursive-family} contains all possible complete intersection numerical semigroups with the property that the bound in Proposition \ref{lower-bound-c-ci} becomes an equality.
\end{example}

\begin{proposition}\label{upper-bound-gen-complete-intersections}
Let $\Gamma$ be a complete intersection numerical semigroup other than $\mathbb N$, minimally generated by $\{r_0,\ldots, r_h\}$. 
If $\mathrm m(\Gamma)\neq 2$, for all  $k$, $r_k<\mathrm F(\Gamma)$.
\end{proposition}
\begin{proof}
Assume without loss of generality that $r_0=\mathrm m(\Gamma)$. The numerical semigroup $\Gamma$ is symmetric and thus for every $i>0$, $\mathrm F(\Gamma)+r_0- r_i\in \Gamma$. If $r_k>\mathrm F(\Gamma)$, for some $k>0$, then $\mathrm F(\Gamma)+r_0-r_k<r_0$, which forces $\mathrm F(\Gamma)+r_0=r_k$. 

If $h>1$, choose $0<i\neq k$. Then $r_k-r_i=\mathrm F(\Gamma)+r_0-r_i\in \Gamma$, contradicting that $r_k$ is a minimal generator. This proves $r_k<\mathrm F(\Gamma)$, whenever $h>1$.

For $h=1$, $\mathrm F(\Gamma)=(r_0-1)(r_1-1)-1$. In this setting, $\Gamma=\langle 2,f+2\rangle$ has $\mathrm F(\Gamma)=f$. For $\mathrm m(\Gamma)>2$, we get $\mathrm F(\Gamma)=(r_0-1)(r_1-1)-1\ge 2(r_1-1)-1 =(r_1-1)+(r_1-2)\ge r_1$.
\end{proof}

\begin{remark}\label{some-facts-complete-intersections}
 If we want to compute the set of all complete intersection numerical semigroups with Frobenius number $f$, then we can use the formula given in Proposition \ref{frob-gluing}. Hence $f=d_1 f_1+d_2 f_2+ d_1 d_2$, and we recursively construct all possible numerical semigroups with Frobenius number $f_1$, and then the set with Frobenius number $f_2$. We next give some useful bounds and facts to perform this task. Denote $f+1$ by $c$.

 \begin{enumerate}[i)]
  \item $d_1\neq 1\neq d_2$. This is because $d_1\in \Gamma_2$ and it is not a minimal generator of $\Gamma_2$. The only possibility to have $d_1=1\in \Gamma_2$ would be $\Gamma_2=\mathbb N=\langle 1\rangle$. But then $d_1$ would be a minimal generator. The same argument is valid for $d_2$.

  \item Since $\gcd(d_1,d_2)=1$, we can assume without loss of generality that $2\le d_2 < d_1$.

  \item Since $f_1,f_2\ge -1$, $f\ge -d_1-d_2+d_1d_2= (d_1-1)(d_2-1)-1$. Hence  $d_2\le \frac{c}{d_1-1}+1$;
  and consequently, $d_2\le \min\{ d_1-1,  \frac{c}{d_1-1}+1\}$. 
  \item $f-d_jf_j\equiv 0\bmod d_i$, $\{i,j\}=\{1,2\}$. In  particular, if $f_j=-1$, then $f+d_j\equiv 0 \bmod d_i$.
  \item $d_1<f$, except in the case $\Gamma= \langle 2=d_2, f+2=d_1\rangle$. 
  \begin{enumerate}[a)]
   \item If $f_1=f_2=-1$, then $\Gamma_1=\Gamma_2=\mathbb N$, and  $\Gamma$ is $\langle d_2, d_1\rangle$. If $d_2\neq 2$, then Proposition \ref{upper-bound-gen-complete-intersections}, asserts that $d_1<f$.
   \item If $f_2>0$, then $f\ge -d_1+d_2+d_1d_2=(d_1+1)(d_2-1)+1\ge d_1+2$. Hence $d_1\le f-2$.
   \item If $f_1>0$, then $f\ge d_1-d_2+d_1d_2=(d_1-1)(d_2+1)+1> 3(d_1-1)\ge d_1+ 2(d_1-1)-1> d_1$.
  \end{enumerate}
  \item If $f_1\neq -1\neq f_2$, then $f-d_1d_2\in \langle d_1,d_2\rangle$. We are only interested in factorizations $f-d_1d_2=a_1d_1+a_2d_2$, $a_1,a_2\in \mathbb N$, with $a_1\equiv a_2\equiv 1 \bmod 2$, since the Frobenius number of a complete intersection is an odd integer.
 \end{enumerate}
\end{remark}

\begin{example}
 We compute the set of all complete intersection numerical semigroups with Frobenius number 11. First note that $\langle 2,13\rangle$ is in this set. The possible $d_1$ belong to $\{3,\ldots,10\}$. \renewcommand{\labelitemii}{$\star$}
 \begin{itemize}
  \item $d_1=10$. Then $2\le d_2 \le \min \{9,\lfloor \frac{12}9\rfloor+1\}=2$. Hence $d_2$ must be 2, but then $\gcd(d_1,d_2)\neq 1$, and we have no complete intersections under these conditions.
  \item $d_1=9$. Then $2\le d_2 \le \min \{8,\lfloor \frac{12}8\rfloor+1\}=2$. This forces $d_2=2$, which in addition is coprime with 9.
  \begin{itemize}
   \item $11+9\equiv 0\bmod 2$, and thus $f_1=-1$ ($\Gamma_1=\mathbb N$) is a possible choice. In this setting $f_2=(11-18+0)/2=1$, whence $\Gamma_2=\langle 2,3\rangle$. We obtain a new complete intersection $\Gamma=9\mathbb N+2\langle 2,3\rangle= \langle 4,6,9\rangle$, because $9\in \langle 2,3\rangle$ is not a minimal generator.
   \item $11+2\not\equiv 0\bmod 9$, so $f_2$ cannot be $-1$.
   \item $11-18\not \in\langle 2,9\rangle$, so we have no more complete intersections with this data.
  \end{itemize}
  \item For $d_1=8$, we have $2\le d_2 \le \min \{7,\lfloor \frac{12}7\rfloor+1\}=2$. However $\gcd\{d_1,d_2\}\neq 1$. 
  \item If $d_1=7$, then $2\le d_2 \le \min \{6,\lfloor \frac{12}6\rfloor+1\}=3$.
    \begin{itemize}
     \item $d_2=2$.
      \begin{itemize}
       \item $11+7\equiv 0\bmod 2$, and thus $\Gamma_1$ can be $\mathbb N$. But then $f_2=(11-14+7)/2=2$, which is even. So this case cannot occur.
       \item $11+2\not\equiv 0\bmod 7$, and so $\Gamma_2$ will not be $\mathbb N$.
       \item Finally, $11-14\not \in \langle 2,7\rangle$, so no complete intersections can be found with properties.
      \end{itemize}
     \item $d_2=3$.
      \begin{itemize}
       \item $11+7\equiv 0\bmod 3$, and thus $\Gamma_1$ could be $\mathbb N$. In this setting $f_2=(11-21+7)/3=-1$, and so $\Gamma_2$ is also $\mathbb N$. We get a new complete intersection $\Gamma=7\mathbb N+3\mathbb N=\langle 3,7\rangle$ with Frobenius number 11.
       \item $11-21\not\in\langle 3,7\rangle$, so no more complete intersections are obtained for this choice of $d_1$ and $d_2$.
      \end{itemize}
    \end{itemize}
  \item For $d_1=6$, $2\le d_2 \le \min \{5,\lfloor \frac{12}5\rfloor+1\}=3$, but both 2 and 3 are not coprime with 6.
  \item $d_1=5$. Then $d_2\in \{2,3,4\}$.
   \begin{itemize}
    \item $d_2=2$.
    \begin{itemize}
     \item $11+5\equiv 0\bmod 2$, and so $\Gamma_1$ can possibly be $\mathbb N$. Hence $f_2=(11-10+5)/2=3$. The only possible complete intersection numerical semigroup with Frobenius number $3$ is $\langle 2,5\rangle$. But $5$ is a minimal generator of this semigroup.
     \item $11+2\not\equiv 0\bmod 5$.
     \item $11-10\not\in\langle 2,5\rangle$.  
    \end{itemize}
    \item $d_2=3$. In this case $11+5\not\equiv 0\bmod 3$, $11+3\not\equiv 0\bmod 5$, and $11-15\not \in\langle 3,5\rangle$.
    \item $d_2=4$. 
    \begin{itemize}
     \item $11+5\equiv 0\bmod 4$, and $f_2=(11-20+5)/4=-1$. So $\Gamma=5\mathbb N+4\mathbb N=\langle 4,5\rangle$ is another complete intersection with Frobenius number 11.
     \item $11-20\not\in \langle 4,5\rangle$.
    \end{itemize}
   \end{itemize}
   \item $d_1=4$, $2\le d_2 \le \min \{3,\lfloor \frac{12}3\rfloor+1\}=3$, and as $\gcd(2,4)\neq 1$, we get $d_2=3$.
   \begin{itemize}
    \item $11+4\equiv 0 \bmod 3$. So $\Gamma_1$ could be $\mathbb N$. If this is the case, $f_2=(11-12+4)/3=1$, which forces $\Gamma_2$ to be $\langle 2,3\rangle$, and $4\in \Gamma_2$ is not a minimal generator. So we obtain $\Gamma=4\mathbb N+3\langle 2,3\rangle= \langle 4,6,9\rangle$, which was already computed before.
    \item $11+3\not\equiv 0\bmod 4$.
    \item $11-12\not\in \langle 3,4\rangle$
   \end{itemize}
   \item $d_3=3$ and $d_2=2$. 
   \begin{itemize}
    \item $11+3\equiv 0\bmod 2$, and $\Gamma_1=\mathbb N$ can be a possibility. Then $f_2=(11-6+3)/2=7$. If we apply this procedure recursively for $f=7$, we obtain that $\{\langle 2,9\rangle, \langle 3,5\rangle, \langle 4,5,6\rangle\}$ is the set of all possible complete intersection numerical semigroups with Frobenius number 7. However, $3\not \in \langle 2,9\rangle$, $3$ is a minimal generator of $\langle 3,5\rangle$, and $3\not\in \langle 4,5,6\rangle$.
    \item $11+2\not\equiv 0\bmod 3$.
    \item $11-6=5\in \langle 2,3\rangle$, and $5=1\cdot 2+1\cdot 3$ is the only factorization. So the only possible choice for $f_1$  and $f_2$ is 1. This means that $\Gamma_1$ and $\Gamma_2$ must be $\langle 2,3\rangle$. Again we obtain no new semigroups, since $2$ and $3$ are minimal generators of $\langle 2,3\rangle$.
   \end{itemize}
 \end{itemize}
Thus the set of complete intersection numerical semigroups with Frobenius number 11 is 
\[\{\langle 2,13\rangle, \langle 4,6,9\rangle, \langle 3,7\rangle, \langle 4,5\rangle \}.\]
\end{example}

\section{Free numerical semigroups}

Throughout this section, let $\Gamma$ be the numerical semigroup $\Gamma$ minimally generated by $\{r_0,\ldots, r_h\}$. For $k\in\{1,\ldots,h+1\}$, set $d_k=\gcd(\{r_0,\ldots,r_{k-1}\})$ ($d_1=r_0$). 

Write $\Gamma_k=\left\langle {\frac{r_0}{d_{k+1}}},\ldots,{\frac{r_k}{d_{k+1}}}\right\rangle$, and $c_k=\mathrm c(\Gamma_k)$ for all $k\in \{1,\ldots,h\}$. Set $c=c_h=\mathrm c(\Gamma)$.

We say that $\Gamma$  is \emph{free} if either $h=0$ (and thus $r_0=1$) or $\Gamma$ is the gluing of the free numerical semigroup $\Gamma_{h-1}$ and $\mathbb N$. Free numerical semigroups were introduced in \cite{bertin}. For other characterizations and properties of free numerical semigroups see \cite[Section 8.4]{ns-book}.

\begin{example}
Notice that the order in which the generators are given is crucial. For instance, $S=\langle 8,10,9\rangle$ is free for the arrangement $(8,10,9)$  but it is not free for $(8,9,10)$. And a numerical semigroup can be free  for different arrangements, for example, $S=\langle 4,6,9\rangle$ has this property.

If we take $c_0,\ldots, c_h$ pairwise coprime integers greater than one, and $r_i= \prod_{j=0, i\neq j}^h c_j$, $j=0,\ldots, h$, then the numerical semigroup generated by $\{r_0,\ldots, r_h\}$ is free for any arrangement of its minimal generating set (see \cite{single-betti}).  
\end{example}

According to Proposition \ref{frob-gluing}, with $A_2=\{r_h\}$, we obtain the following consequence.
\begin{corollary}\label{frob-free}
If $\Gamma$ is free, then
\[\mathrm F(\Gamma)= d_h \mathrm F(\Gamma_{h-1})+r_h(d_h-1).\]
\end{corollary}
In this way we retrieve Johnson's formula (\cite{johnson}). Notice also that $\Gamma_{h-1}$ is again free, so if we expand recursively this formula we obtain the formula given by Bertin and Carbonne for free numerical semigroups (see \cite{bertin}; these authors named these semigroups in this way).

This equation can be reformulated in terms of the conductor as 
\begin{equation}\label{conductor-free-gluing} 
\mathrm c(\Gamma)=c_h=d_h c_{h-1}+(d_h-1)(r_h-1).
\end{equation}

\begin{lemma} \label{some-facts-free}
If $\Gamma$ is free, then
\begin{enumerate}
\item $\gcd(d_h,r_h)=1$;
\item $d_h\mid \mathrm F(\Gamma)+r_h$ (consequently $d_h \not| \, \mathrm F(\Gamma)$);

\item if we define $e_k={\frac{d_k}{d_{k+1}}},k=1,\ldots,h$, then
$
e_k {r}_k\in \langle r_0,\ldots,r_{k-1}\rangle$,
for all $k=1,\ldots,h$; in particular, $e_k\ge 2$;

\item  $d_1>d_2>\cdots>d_{h+1}=1$; 

\item  $d_h\le \frac{\mathrm c(\Gamma)}{r_h-1}+1$;

\item for $h\ge 1$, $(d_h-1)(r_h-1)\ge 2^h$.
\end{enumerate}
\end{lemma}
\begin{proof}
\begin{enumerate}
\item This follows from  the fact that  $\Gamma$ is a numerical semigroup, and thus $\gcd(d_h,r_h)=d_{h+1}=1$.

\item $\mathrm F(\Gamma_{h-1})=(\mathrm F(\Gamma)+r_h(1-d_h))/d_h=(\mathrm F(\Gamma)+r_h)/d_h-1$.

\item As $\Gamma_k$ is the gluing of $\Gamma_{k-1}$ and $\mathbb N$, we have that $\frac{r_k}{d_{k+1}}\in \Gamma_{k-1}$. Hence $\frac{d_k}{d_{k-1}}r_k\in \langle r_0,\ldots,r_{k-1}\rangle$. If $e_k=1$, then $r_k\in \langle r_0,\ldots,r_{k-1}\rangle$, contradicting that $r_k$ is a minimal generator.

\item By definition, $d_k\ge d_{k+1}$. As $e_k=\frac{d_k}{d_{k+1}}\ge 2$, we get $d_k>d_{k-1}$.

\item Notice that $\mathrm F(\Gamma)\ge (r_h-1)(d_h-1) -1$, since $\mathrm F(\Gamma_{h-1})\ge -1$.

\item If $d_h=2$, then we show that $r_h>\mathrm m(\Gamma)$. Assume to the contrary that $r_h=\mathrm m(\Gamma)$. Then we already proved above that $e_h r_h\in \langle r_0,\ldots,r_{h-1}\rangle$. Since $e_h=d_h$ and $r_i$ is a minimal generator of $\Gamma$ for all $i$, we deduce that $2 r_h = \sum_{i=0}^{h-1}a_i r_i$, with $\sum_{i=0}^{h-1}a_i\ge 2$. As $r_i>r_h$ for every $i=0,\ldots,h-1$, we get $2 r_h > r_h\sum_{i=0}^{h-1} a_i$, and thus $\sum_{i=0}^{h-1}a_i<2$, a contradiction. Thus in view of  Proposition \ref{lower-bound-m-complete-intersection}, we have that $r_h\ge 2^h$, and if $d_h=2$, then $r_h\ge 2^h+1$. Hence for $d_h=2$ the proof follows easily, and for $d_h>2$ we get $(d_h-1)(r_h-1)\ge 2(r_h-1)\ge 2(2^h-1)\ge 2^h$ (we are assuming $h\ge 1$).
\end{enumerate}
\end{proof}

In view of Example \ref{recursive-family-ii}, the bound proposed in Proposition \ref{lower-bound-c-ci} cannot be improved for free numerical semigroups, since the family introduced in Example \ref{example-recursive-family} consists on free numerical semigroups. However, we can use Proposition \ref{lower-bound-c-ci} to find an upper bound for $r_h$, as we show next.

For all $h\geq 2$, $c_{h-1}=\frac{c-(d_h-1)(r_h-1)}{d_h}$ is an even integer, and $c=c_h\geq h 2^h$. In particular,
$$
-c_{h-1}d_h \leq -(h-1) 2^{h-1}d_h.
$$
Hence
$$
(r_h-1)(d_h-1)\leq c-(h-1)2^{h-1}d_h
$$

This gives us the following upper bound for $r_h$.
$$
r_h\leq \frac{c}{d_h-1}-(h-1)2^{h-1}\frac{d_h}{d_h-1}+1.
$$

\begin{corollary}\label{bound-rh-free}
For all $h\geq 2$, 
$$
2^{h}+1\leq r_h \leq   \frac{c}{d_h-1}-(h-1)2^{h-1}\frac{d_h}{d_h-1}+1 \le c -(h-1)2^{h-1}+1.
$$
\end{corollary}

%

\begin{remark}
In order to compute the set of all free numerical semigroups with a given Frobenius number, we make use of the formula given in Corollary \ref{frob-free}, by taking into account the restrictions given in this section for $d_h$  and $r_h$.
\end{remark}

\section{Telescopic numerical semigroups}

We keep using the same notation as in the preceding section.  We say that the numerical semigroup $\Gamma$ minimally generated by $\{r_0,\ldots, r_h\}$ is \emph{telescopic} if it is free for the arrangement of the generators $r_0<\cdots<r_h$ (see for instance \cite{telescopic}). This motivates the notation $\{r_0<\cdots<r_h\}$, that means that the elements in the set $\{r_0,\ldots,r_h\}$ fulfill the extra condition $r_0<\cdots <r_h$. We will also write $\Gamma=\langle r_0<\cdots <r_h\rangle$ when $\{r_0,\ldots,r_h\}$ is a generating system for $\Gamma$ and $r_0<\cdots <r_h$.


Notice that in addition to the properties we had for free numerical semigroups, if $\Gamma$ is telescopic, then
\begin{enumerate}
\item $d_h<r_h$, because $d_h\mid r_{h-1}<r_h$;
\item  $\mathrm F(\Gamma)\ge (r_h-1)(d_h-1) -1>(d_h-1)^2-1$, whence $d_h\le \min\left\{r_h-1,\frac{\mathrm c(\Gamma)}{r_h-1}+1,\sqrt{\mathrm c(\Gamma)}+1\right\}$.
\end{enumerate}

\begin{proposition}\label{lower-bound-rh-telescopic}
Let $\Gamma$ be a telescopic numerical semigroup minimally generated by $\{r_0<\cdots < r_h\}$. If $h\geq 2$, then $r_h\geq 2^{h+1}-1$.
\end{proposition}
\begin{proof} Let $h=2$, and let $\Gamma_1=\left\langle \frac{r_0}{d_2},\frac{r_1}{d_2}\right\rangle$. Since $\frac{r_1}{d_2}\geq 3$ and $d_2\geq 2$, we have $r_1\geq 6$. Besides, $r_2> r_1$, whence $r_2\geq 7$. Note that this bound is attained for $\Gamma_2=\langle 4,6,7\rangle$.

Assume that $h\geq 3$, and that the formula is true for $h-1$. We have $r_h \geq
r_{h-1}+1$ and $r_h\in \left\langle \frac{r_0}{d_h},\ldots,\frac{r_{h-1}}{d_h}\right\rangle$. By induction hypothesis, 
we have $\frac{r_{h-1}}{d_h}\geq 2^{h}-1$. Hence $r_h\geq 2 (2^{h}-1)+1=2^{h+1}-1$. Note that this bound is reached by $\Gamma_h=\langle 2^{h},3\cdot2^{h-1},7\cdot 2^{h-2},\ldots,(2^{k+1}-1)\cdot 2^{h-k},\ldots,2^{h+1}-1\rangle$.
\end{proof}

As in the free case, we can describe a bound for the embedding dimension of a telescopic numerical semigroup.

\begin{proposition} \label{lower-bound-c-telescopic} 
Let $\Gamma$ be a telescopic numerical semigroup other than $\mathbb N$. Then \[(\mathrm e(\Gamma)-2)2^{\mathrm e(\Gamma)}+2 \le \mathrm c(\Gamma).\]
\end{proposition}

\begin{proof}  
Assume that $\Gamma$ is minimally generated by $\{r_0<\cdots< r_h\}$. Denote as usual $\mathrm c(\Gamma)$ by $c$. We use once more induction on $h$. 


The case $h=1$ is evident.


Suppose  that $h\geq 2$, and that our inequality is true for $h-1$. By (\ref{conductor-free-gluing}), we have $c=d_h c_{h-1}+(d_h-1)(r_h-1)$. By induction hypothesis, $c_{h-1}\geq (h-2)2^h+2$,  and as $d_h\geq 2$, and $r_h\geq 2^{h+1}-1$, we get $c\geq (h-2)2^{h+1}+4+2^{h+1}-2=(h-1)2^{h+1}+2$. 
\end{proof}

Note that for all $h\geq 2$,
$c_{h-1}={\displaystyle{\frac{c-(d_h-1)(r_h-1)}{d_h}}}$ is an even integer, and that $c=c_h\geq (h-1)2^{h+1}+2$. In particular
$$
-c_{h-1}d_h \leq -\big((h-2)2^h+2\big)d_h.
$$
Hence
$$
(r_h-1)(d_h-1)\leq c-\big((h-2)2^h+2\big)d_h.
$$
This gives us  the following upper bound for $r_h$:
$$
r_h\leq {\frac{c}{d_h-1}}-\big((h-2)2^h+2\big){\frac{d_h}{d_h-1}}+1\le c-(h-2)2^h-1. 
$$

\begin{corollary}\label{bounds-rh-telescopic} For all $h\geq 2$, we have
$$
\quad 2^{h+1}-1\leq
r_h \le  {\frac{c}{d_h-1}}-\big((h-2)2^h+2\big)\frac{d_h}{d_h-1}+1 \le c-(h-2)2^h-1.
$$
\end{corollary}


\begin{remark}
For computing the set of all telescopic numerical semigroups with fixed Frobenius number, we proceed as in the free case, ensuring that $r_h$ is larger than the largest generator of $\Gamma_1$ multiplied by $d_h$. Notice that $d_h$ must now be smaller than $r_h$. 
\end{remark}

\section{Plane curve singularities}

Let  $\Gamma$ be the numerical semigroup minimally generated by $\{r_0 <r_1<\ldots <r_h\}$. Let $d_k$, $\Gamma_k$, $c_k$, and $e_k$ be as in the preceding section. The numerical semigroup $\Gamma$ is the numerical semigroup associated to an irreducible plane curve singularity if $\Gamma$ is telescopic and $e_kr_k<r_{k+1}$ for all $k=1,\ldots,h-1$ (see \cite{zar}).

\begin{proposition}\label{lower-bound-rh-planar}
Let $\Gamma$ be the semigroup associated to an irreducible plane curve singularity minimally generated by $\{r_0<\cdots <r_h\}$, with $h\ge 2$. Then $r_h\geq  \frac{1}{3}(5\cdot 2^{2h-1}-1)$.
\end{proposition}

\begin{proof} 
For $h=2$, as $\Gamma_1=\left\langle \frac{r_0}{d_2}< \frac{r_1}{d_2}\right\rangle$,  we obtain $\frac{r_1}2\geq 3$. Since $d_2\geq 2$, we deduce that $r_1\geq 6$. The plane singularity condition implies $r_2> e_1r_1\geq 12$, because we know that $e_1\ge 2$ (Lemma \ref{some-facts-free}). Hence $r_2\ge 13$.

Assume that $h\geq 3$, and that the formula is true for $h-1$. The plane singularity condition for $k=h-1$ implies that $r_h \geq \frac{r_{h-1}}{d_h}d_{h-1}+1$. The quotient $\frac{r_{h-1}}{d_h}$ is the largest generator of $\Gamma_{h-1}$. The induction hypothesis then asserts that $\frac{r_{h-1}}{d_h} \geq  \frac{1}{3}(5\cdot 2^{2(h-1)-1}-1)$. By using that $e_k\ge 2$ for all $k$ (Lemma \ref{some-facts-free}), we deduce that $d_{h-1} \geq 4$. By putting all this together, we get  $r_h\geq
4(\frac{1}{3}(5\cdot 2^{2(h-1)-1}-1))+1=\frac{1}{3}(5\cdot 2^{2h-1}-1)$.
\end{proof}

\begin{proposition} \label{lower-bound-c-planar} 
Let $\Gamma\neq \mathbb N$ be the semigroup associated to an irreducible plane curve singularity minimally generated by $\{r_0<\cdots <r_h\}$  and with conductor $c$. Then
\[c\geq 
\frac{5}{3}2^{2h}-3\cdot 2^h+\frac{4}{3}.\] 
\end{proposition}

\begin{proof} 
The case $h=1$ is evident.

Assume that $h\geq 2$ and that our inequality holds for $h-1$. We have: $c=d_h c_{h-1}+(d_h-1)(r_h-1)$. By induction hypothesis $c_{h-1}\geq 
\frac{5}{3}2^{2h-2}-3\cdot 2^{h-1}+\frac{4}{3}$. Notice that $d_h\geq 2$. Thus our assertion follows from Proposition \ref{lower-bound-rh-planar}.
\end{proof}


We proceed now as we did in the telescopic case to obtain also an upper bound for $r_h$. 
Note that for all $h\geq 2$, $c_{h-1}={{\frac{c-(d_h-1)(r_h-1)}{d_h}}}$ is an even integer, and that $c_{h-1}\geq {\frac{5}{3}}2^{2h-2}-3\cdot 2^{h-1}+
{\frac{4}{3}}$. Thus
$$
-c_{h-1}d_h \leq -\left({\frac{5}{3}}2^{2h-2}-3\cdot 2^{h-1}+
{\frac{4}{3}}\right)d_h.
$$
Hence
$$
(r_h-1)(d_h-1)\leq c-\left({\frac{5}{3}}2^{2h-2}-3\cdot 2^{h-1}+
{\frac{4}{3}}\right)d_h.
$$
This gives us  the following upper bound for $r_h$.
$$
r_h\leq \frac{c}{d_h-1}-\left(\frac{5}{3}2^{2h-2}-3\cdot 2^{h-1}+
{\frac{4}{3}}\right)\frac{d_h}{d_h-1}+1.
$$

\begin{corollary}\label{bounds-rh-planar} 
For all $h\geq 2$, we have
$$
\frac{5}{3}2^{2h-1}-\frac{1}{3}\le
r_h \le \frac{c}{d_h-1}-\left(\frac{5}{3}2^{2h-2}-3\cdot 2^{h-1}+
\frac{4}{3}\right)\frac{d_h}{d_h-1}+1 \le c- \frac{5}{3}2^{2h-2}-3
\cdot 2^{h-1}+
\frac{7}{3}.
$$
\end{corollary}

A bound for the embedding dimension also follows from the above proposition.
\begin{corollary}
If $h\ge 2$, then  
\[h\le \log_2\left(\frac{\sqrt{60\,c+1}+9}{10}\right).\]
\end{corollary}
\begin{proof}
From Proposition \ref{lower-bound-c-planar}, $\frac{5}{3}2^{2h}-3\cdot 2^h+\frac{4}{3}\le c$. Write $x=2^h$, we get $5/3x^2-3 x+\frac{4}{3}\le c$. By solving $5/3x^2-3 x+\frac{4}{3}- c=0$, we get $x\in \left\{-\frac{\sqrt{60\,c+1}-9}{10},\frac{\sqrt{60\,c+1}+9}{10}\right\}$. As the minimum of $5/3x^2-3 x+\frac{4}{3}- c$ is reached in $x=9/10$, and in our setting $x=2^h>1$, we have that  the maximum possible $x>0$ such that $\frac{5}{3}2^{2h}-3\cdot 2^h+\frac{4}{3}\le c$ is $x=\frac{\sqrt{60\,c+1}+9}{10}$.
\end{proof}

\begin{remark}
The set of all numerical semiogrups with fixed Frobenius number associated to an irreducible planar curve singularity is calculated as in the free case, by imposing the condition $e_k r_k<r_{k+1}$.
\end{remark}

\section{Experimental results}

With the ideas given in the preceding sections, we implemented in \texttt{GAP} (\cite{gap}), with the help of the \texttt{numericalsgps} package (\cite{numericalsgps}), functions to compute the set of all complete intersection, free and telescopic numerical semigroups, as well as the set of all numerical semigroups associated to irreducible planar curve singularities with fixed Frobenius number (these functions will be included in the next release of this package).

The following table was computed in  6932 milliseconds on a 2.5GHz desktop computer, and it shows, for fixed genus g, the number of complete intersections (ci(g)), free (fr(g)), telescopic (tl(g)), associated to an irreducible planar curve singularity (pc(g)) numerical semigroups, respectively. Recall that for a symmetric numerical semigroup its conductor is twice its genus.

Observe that almost all complete intersections in this table are free. This is due to the fact that the embedding dimension of all numerical semigroups appearing there is small,  and for embedding dimension three or less, the concepts of free and complete intersections coincide (among the complete intersection numerical semigroups represented in the table 158 of them have embedding dimension 2, 1525 have embedding dimension 3, 1862 have embedding dimension 4, and 205 have embedding dimension 5).

\medskip

\begin{tabular}{|l|l|l|l|l||l|l|l|l|l||l|l|l|l|l|}
g & ci(g) & fr(g) & tl(g) & pc(g) & g & ci(g) & fr(g) & tl(g) & pc(g) &
g & ci(g) & fr(g) & tl(g) & pc(g) \\ \hline
0 & 1 & 1 & 1 & 1 &  19 & 24 & 24 & 12 & 5 & 38 & 61 & 61 & 37 & 12 \\
1 & 1 & 1 & 1 & 1 &  20 & 16 & 16 & 11 & 6 &  39 & 100 & 100 & 52 & 16 \\
2 & 1 & 1 & 1 & 1 &  21 & 27 & 27 & 18 & 9 &  40 & 110 & 109 & 54 & 19 \\
3 & 2 & 2 & 2 & 2 &  22 & 31 & 31 & 19 & 8 &  41 & 80 & 79 & 47 & 12 \\
4 & 3 & 3 & 2 & 2 &  23 & 21 & 21 & 13 & 6 & 42 & 122 & 120 & 61 & 20 \\
5 & 2 & 2 & 2 & 1 &  24 & 36 & 35 & 20 & 11 &  43 & 120 & 120 & 60 & 17 \\
6 & 4 & 4 & 4 & 3 &  25 & 38 & 38 & 22 & 9 &  44 & 94 & 94 & 48 & 15 \\
7 & 5 & 5 & 3 & 2 &  26 & 27 & 27 & 16 & 8 &  45 & 143 & 142 & 73 & 22 \\
8 & 3 & 3 & 2 & 2 &  27 & 46 & 46 & 24 & 11 & 46 & 151 & 149 & 72 & 21 \\
9 & 7 & 7 & 5 & 4 &  28 & 45 & 45 & 25 & 10 & 47 & 108 & 106 & 57 & 15 \\
10 & 8 & 8 & 6 & 4 & 29 & 34 & 33 & 20 & 7 &  48 & 158 & 157 & 75 & 24 \\
11 & 5 & 5 & 4 & 2 &  30 & 57 & 57 & 32 & 13 &  49 & 179 & 179 & 84 & 23 \\
12 & 11 & 11 & 8 & 5 &  31 & 62 & 62 & 31 & 9 & 50 & 128 & 128 & 68 & 20 \\
13 & 11 & 11 & 8 & 3 & 32 & 43 & 43 & 25 & 10 &  51 & 197 & 194 & 86 & 26 \\
14 & 9 & 9 & 7 & 4 &  33 & 65 & 65 & 37 & 14 &  52 & 209 & 207 & 89 & 27 \\
15 & 14 & 14 & 10 & 6 & 34 & 77 & 76 & 39 & 13 &  53 & 142 & 142 & 76 & 20 \\
16 & 17 & 17 & 9 & 5 & 35 & 53 & 52 & 29 & 11 & 54 & 229 & 227 & 101 & 30 \\
17 & 12 & 12 & 8 & 3 &  36 & 84 & 83 & 43 & 17 &  55 & 238 & 235 & 104 & 29 \\
18 & 18 & 18 & 12 & 6 &  37 & 90 & 90 & 47 & 13 &  56 & 172 & 169 & 83 & 24 \\ \hline 
\end{tabular}

\medskip

The largest genus, for which the set of numerical semigroups with this genus is known, is 55, and the number of numerical semigroups with genus 55 is 1142140736859 (\cite{genus}), while there are just 2496 symmetric numerical semigroup with genus 55 (this last amount can be computed by using the \texttt{Irreducible\-Numerical\-Semigroups\-With\-Frobenius\-Number} command of the \texttt{numericalsgps} package). The proportion of complete intersections among symmetric numerical semigroups is small, and tiny compared with the whole set of numerical semigroups.

\medskip

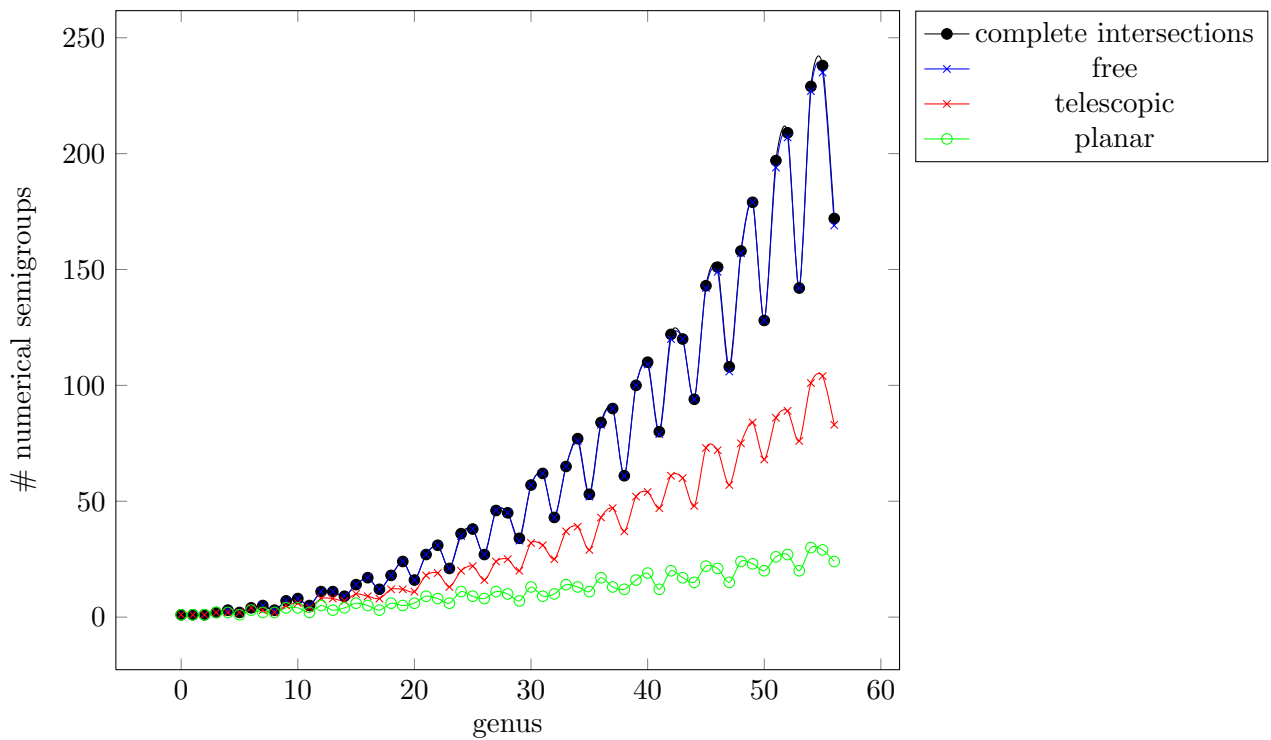
\begin{tikzpicture}
\pgfplotsset{every axis legend/.append style={
at={(1.02,1)},
anchor=north west}}

    \begin{axis}[
	width=12cm,
        xlabel=genus,
        ylabel=\# numerical semigroups
]
    \addplot[smooth,mark=*,black] plot coordinates {
(0,1)
(1,1)
(2,1)
(3,2)
(4,3)
(5,2)
(6,4)
(7,5)
(8,3)
(9,7)
(10,8)
(11,5)
(12,11)
(13,11)
(14,9)
(15,14)
(16,17)
(17,12)
(18,18)
(19,24)
(20,16)
(21,27)
(22,31)
(23,21)
(24,36)
(25,38)
(26,27)
(27,46)
(28,45)
(29,34)
(30,57)
(31,62)
(32,43)
(33,65)
(34,77)
(35,53)
(36,84)
(37,90)
(38,61)
(39,100)
(40,110)
(41,80)
(42,122)
(43,120)
(44,94)
(45,143)
(46,151)
(47,108)
(48,158)
(49,179)
(50,128)
(51,197)
(52,209)
(53,142)
(54,229)
(55,238)
(56,172)
    };
    \addlegendentry{complete intersections}

    \addplot[smooth,mark=x,blue] plot coordinates {
(0,1)
(1,1)
(2,1)
(3,2)
(4,3)
(5,2)
(6,4)
(7,5)
(8,3)
(9,7)
(10,8)
(11,5)
(12,11)
(13,11)
(14,9)
(15,14)
(16,17)
(17,12)
(18,18)
(19,24)
(20,16)
(21,27)
(22,31)
(23,21)
(24,35)
(25,38)
(26,27)
(27,46)
(28,45)
(29,33)
(30,57)
(31,62)
(32,43)
(33,65)
(34,76)
(35,52)
(36,83)
(37,90)
(38,61)
(39,100)
(40,109)
(41,79)
(42,120)
(43,120)
(44,94)
(45,142)
(46,149)
(47,106)
(48,157)
(49,179)
(50,128)
(51,194)
(52,207)
(53,142)
(54,227)
(55,235)
(56,169)
    };
    \addlegendentry{free}

    \addplot[smooth,color=red,mark=x]
        plot coordinates {
(0,1)
(1,1)
(2,1)
(3,2)
(4,2)
(5,2)
(6,4)
(7,3)
(8,2)
(9,5)
(10,6)
(11,4)
(12,8)
(13,8)
(14,7)
(15,10)
(16,9)
(17,8)
(18,12)
(19,12)
(20,11)
(21,18)
(22,19)
(23,13)
(24,20)
(25,22)
(26,16)
(27,24)
(28,25)
(29,20)
(30,32)
(31,31)
(32,25)
(33,37)
(34,39)
(35,29)
(36,43)
(37,47)
(38,37)
(39,52)
(40,54)
(41,47)
(42,61)
(43,60)
(44,48)
(45,73)
(46,72)
(47,57)
(48,75)
(49,84)
(50,68)
(51,86)
(52,89)
(53,76)
(54,101)
(55,104)
(56,83)
};
    \addlegendentry{telescopic}

\addplot[smooth,mark=o,green] plot coordinates {
(0,1)
(1,1)
(2,1)
(3,2)
(4,2)
(5,1)
(6,3)
(7,2)
(8,2)
(9,4)
(10,4)
(11,2)
(12,5)
(13,3)
(14,4)
(15,6)
(16,5)
(17,3)
(18,6)
(19,5)
(20,6)
(21,9)
(22,8)
(23,6)
(24,11)
(25,9)
(26,8)
(27,11)
(28,10)
(29,7)
(30,13)
(31,9)
(32,10)
(33,14)
(34,13)
(35,11)
(36,17)
(37,13)
(38,12)
(39,16)
(40,19)
(41,12)
(42,20)
(43,17)
(44,15)
(45,22)
(46,21)
(47,15)
(48,24)
(49,23)
(50,20)
(51,26)
(52,27)
(53,20)
(54,30)
(55,29)
(56,24)
};
  \addlegendentry{planar}

    \end{axis}
    \end{tikzpicture}

\medskip

Also as one of the referees observed, local minimums in the graph are attained when the genus is congruent with 2 modulo 3. We have not a proof for this fact. May be this behavior is inherited from the symmetric case. The sequence 
\[
\begin{array}{l}
1, 1, 1, 2, 3, 3, 6, 8, 7, 15, 20, 18, 36, 44, 45, 83, 109, 101, 174, 246, 227,420, 546, 498, 926, 1182, 1121,\\ 2015, 2496, 2436, 4350,  5602, 5317,  8925, 11971, 11276,
\end{array}
\]
represents the number of symmetric numerical semigroups with genus ranging from 0 to 35.

The following table shows that the proportion between complete intersections and free numerical semigroups remains similar even for larger genus. Observe that for genus 310 it takes 70 minutes to compute the set of all complete intersections, while it takes approximately 8 minutes and 30 seconds to determine all free numerical semigroups with this genus. For genus 55, computing the set of all numerical semigroups with this genus might take several months and a few terabytes (this was communicated to us by Manuel Delgado, see \cite{genus}).

\medskip

\begin{center}
\begin{tabular}{|l|l|l|l|l|l| }\hline
g & ci(g) & milliseconds &  fr(g) & milliseconds & fr(g)/ci(g) \\ \hline \hline
220 & 18018 & 538213 & 17675 & 94134 & 0.98\\ \hline
230 & 16333 & 660838 & 16026 & 108187 & 0.98 \\ \hline
240 & 24862 & 924409 & 24359 & 153069 & 0.98\\ \hline
250 & 28934 & 1167901 & 28355 & 158706 & 0.98 \\ \hline
260 & 25721 & 1389167 &  25186 & 177691 & 0.98\\ \hline
310 & 66335 & 4206374 & 64959 & 509691 & 0.98\\ \hline
\end{tabular}
\end{center}

\end{document}